\documentclass{amsart}
\usepackage[T1]{fontenc}
\begin{document}

\vspace{0.5in}

\renewcommand{\bf}{\bfseries}
\renewcommand{\sc}{\scshape}
\vspace{0.5in}

\title{On Convex Hulls and the Quasiconvex Subgroups of $F_m\times\mathbb{Z}^n$}

\author{Jordan Sahattchieve}
\address{Department of Mathematics;University of Michigan;
Ann Arbor, Michigan 48109}
\email{jantonov@umich.edu}


\newtheorem{thm}{Theorem}[section]
\newtheorem{lem}[thm]{Lemma}
\newtheorem{pro}[thm]{Proposition}
\newtheorem{cor}[thm]{Corollary}
\newtheorem{ex}[thm]{Example}
\theoremstyle{defn}
\newtheorem{defn}[thm]{Definition}

\newtheorem{remark}[thm]{Remark}
\numberwithin{equation}{section}

\begin{abstract}
In this paper, we explore a method for forming the convex hull of a subset in a uniquely geodesic metric space due to Brunn and use it to show that with respect to the usual action of\\ $F_m\times\mathbb{Z}^n$ on $Tree\times\mathbb{R}^n$, every quasiconvex subgroup of $F_m\times\mathbb{Z}^n$ is convex.  Further, we show that the Cartan-Hadamard theorem can be used to show that locally convex subsets of complete and connected CAT(0) spaces are convex (this part was taken from the author's Ph.D. dissertation submitted in the Spring of 2012, see \cite{Sah2} and \cite{JSD}).  Finally, we show that the quasiconvex subgroups of $F_m\times\mathbb{Z}^n$ are precisely those of the form $A\times B$, where $A\leq F_m$ is finitely generated, and $B\leq\mathbb{Z}^n$.
\end{abstract}

\maketitle

\section{\bf Introduction and Basic Notions}\label{QCIntro}
The motivation for the work in this paper comes from the following remark found in the introductory section of \cite{GH}: "...it is currently unknown whether a quasiconvex subgroup of a CAT(0) group is itself CAT(0)."  At the time of first being acquainted with the problem, it seemed to me rather interesting that the answer to such a fundamental question in the theory of CAT(0) groups was not yet known.  By comparison, the corresponding statement in the theory of hyperbolic groups has long been known to be true.  We begin by recalling some basic definitions:

Let $(X,d)$ be a metric space, $x,y\in X$, and let $\lambda$ be a map from the closed interval $\left[0,l\right]$ to $X$, such that $c(0)=x$, $c(l)=y$, and such that $d(\lambda(t_1),\lambda(t_2))=\left|t_1-t_2\right|$ for all $t_1,t_2\in \left[0,l\right]$.  Then, we say that $\lambda$ is a \textit{geodesic segment} with endpoints $x$ and $y$ and whenever there is no ambiguity, we shall denote this geodesic segment by $\left[x,y\right]$.  To set up notation once and for all, if $\lambda:\left[0,l\right]\rightarrow X$ is the unique geodesic such that $\lambda(0)=x$ and $\lambda(l)=y$, we shall denote $\lambda(t)$ by $\left[x,y\right](t)$; also given any $Y\subseteq X$, $\mathcal{N}_{\nu}(Y)$ will stand for a $\nu$-neighborhood of $Y$, namely the set $\left\{x\in X:d(x,y)<\nu, \:for\: some\: y\in Y\right\}$.  If every two points in $X$ can be connected by a geodesic, we call $X$ a \textit{geodesic metric space}.
\begin{defn}Let $(X,d)$ be a geodesic metric space, and let $a,b,c\in X$.  Let $\Delta(a,b,c)$ denote the geodesic triangle $\left[a,b\right]\cup\left[b,c\right]\cup\left[a,c\right]$ in $X$.  We shall say that three points $\overline{a},\overline{b}$ and $\overline{c}$ in the Euclidean plane determine a comparison triangle $\overline{\Delta}(\overline{a},\overline{b},\overline{c})$ for $\Delta(a,b,c)$, if $d_X(a,b)=d_{\mathbb{E}^2}(\overline{a},\overline{b})$, $d_X(b,c)=d_{\mathbb{E}^2}(\overline{b},\overline{c})$, and $d_X(a,c)=d_{\mathbb{E}^2}(\overline{a},\overline{c})$.  The geodesic triangle $\Delta(a,b,c)$ is said to satisfy the CAT(0) comparison inequality if for any $x,y\in\Delta(a,b,c)$, $d_X(x,y)\leq d_{\mathbb{E}^2}(\overline{x},\overline{y})$, where $\overline{x}$ and $\overline{y}$ are the corresponding comparison points for $x$ and $y$ respectively.
\end{defn}

Now, we are now ready to state the definition of a CAT(0) space:

\begin{defn}\label{CAT}Let $(X,d)$ be a geodesic metric space. If for any three points $a,b,c\in X$, the geodesic triangle $\Delta(a,b,c)$ satisfies the CAT(0) comparison inequality, we shall say that $(X,d)$ is a CAT(0) space.
\end{defn}

The letters $C$, $A$, and $T$ stand for the initials of the last names of the mathematicians Cartan, Alexandrov, and Toponogov, whereas the number $0$ refers to an upper bound for the "curvature" of $X$.  We cannot use the classical geometric notions of curvature in an arbitrary geodesic metric space due to the lack of a differentiable structure.  However, a well-known result in differential geometry, which states that complete, simply connected Riemannian manifolds whose sectional curvature is bounded above by $\kappa<0$ are CAT(0) spaces, allows us to capture the salient features of non-positive curvature and motivates Definition \ref{CAT} above.

Now, given a uniquely geodesic space $X$, we would like to single out an important for us class of subspaces of $X$, namely the convex ones:
\begin{defn}
Let $(X,d)$ be a uniquely geodesic metric space.  A subset $A$ of $X$ is called convex if $\left[a,b\right]\subseteq A$ whenever $a,b\in A$.  
\end{defn}  

Given a subset $Y\subseteq X$, the smallest convex subset which contains it is called the \textit{convex hull} of $Y$, which we will denote by $conv(Y)$.
Next, we would like to quasify the idea of convexity:

\begin{defn}\label{QCSpace}Let $X$ be a uniquely geodesic metric space and let $Y\subseteq X$ be a subspace.  We shall say that $Y$ is $\nu$-quasiconvex if there exists $\nu>0$ such that $\left[x,y\right]\subseteq\mathcal{N}_{\nu}(Y)$, for all $x,y\in Y$.
\end{defn}

The notion of quasiconvexity allows for some wiggle room: a geodesic joining points in $Y$ need not be contained in $Y$ itself but is rather allowed to travel in a fixed bounded neighborhood of $Y$ instead.

Now, we introduce groups into the picture: let $X$ be a set and $G$ a group.  A \textit{group action} is a map $\sigma:G\times X\rightarrow X$, such that $\sigma (e,x)=x$, and $\sigma(g,\sigma(h,x))=\sigma(gh,x)$ for every $x\in X$.  For convenience, we shall just write $g\cdot x$ instead of $\sigma(g,x)$.  If $X$ has the structure of a metric space, we shall only be interested in group actions via isometries.  We say that $G$ acts on $X$ by isometries if $d(g\cdot x,g\cdot y)=d(x,y)$ for all $g\in G$, and all $x,y\in X$.  The action is called \textit{proper} if for each $x\in X$, there exists $r>0$ such that the set $\left\{g\in G:g\cdot B(x,r)\cap B(x,r)\neq\emptyset\right\}$ is finite, and the action is called \textit{cocompact} if there is a compact $K\subseteq X$, such that $X=\bigcup_{g\in G}g\cdot K$.  If a group $G$ acts properly and cocompactly on a metric space $X$, we say that $G$ acts geometrically on $X$.  If $X$ is also a CAT(0) space and $G$ acts geometrically, $G$ is called a \textit{CAT(0) group}.

\begin{defn}\label{QCSubgrpCAT0}Let $G$ be a CAT(0) group acting geometrically on the CAT(0) space $X$, let $H\subseteq G$ be a subgroup, and let $x_0\in X$ be a basepoint.  The subgroup $H$ is called $\nu$-quasiconvex if the group orbit $Hx_0$ is a $\nu$-quasiconvex subspace of $X$.
\end{defn}
Often, we make no mention of the quasiconvexity constant $\nu$ and simply say that $H$ is a \textit{quasiconvex} subgroup.  While the value of the constant $\nu$ in Definition \ref{QCSubgrpCAT0} may depend on the choice of the basepoint $x_0$, whether $H$ is quasiconvex or not does not depend on this choice.

It is important to note here that the notion of quasiconvexity in CAT(0) groups, unlike its counterpart in hyperbolic groups, depends on the choice of action.  As Example 2.10 in \cite{GH} shows, considering two actions of $G=F_2\times\mathbb{Z}$ on the same CAT(0) space, we can arrange a subgroup of $G$ to be quasiconvex with respect to one action but not the other. 

In order to show that a quasiconvex subgroup $H$ of the CAT(0) group $G$ acting geometrically on the CAT(0) space $X$ is a CAT(0) group, we need to exhibit a CAT(0) space $Y$ and a geometric action of $H$ on $Y$.  The reader may at this point rightfully ask the question, why can we not take the convex hull of the $H$-orbit of some point $x_0\in X$?  The convex hull is a convex, $H$-invariant subspace of a CAT(0) space, and the action of $H$ on it is proper, as the action of $G$ on $X$ is proper.  The problem is this: while the action of $G$ on $X$ was assumed to be cocompact, it is not at all obvious, and perhaps not even true in general, that $H$ being quasiconvex in $G$ should imply cocompactness of the action of $H$ on $conv(Hx_0)$.  Of course, even if one were to find a counterexample, that is an example where the induced action of $H$ on $conv(Hx_0)$ is not cocompact, this would not necessarily mean that $H$ is not a CAT(0) group because there is still the possibility that $H$ may act geometrically on some other CAT(0) space.  As far as we know, the general question whether quasiconvexity implies CAT(0) is still wide open.  In this paper, we show that what \textit{should} be true is indeed true in one special case, namely, we prove the following:\\\\
\textbf{Theorem}:  Let $H$ be a quasiconvex subgroup of $G=F_m\times\mathbb{Z}^n$, and let $X$ be the product of the regular $2m$-valent tree with $\mathbb{R}^n$ with the usual action of $G$.  Then the action of $H$ on the convex hull of any orbit $Hx_0$ is cocompact.\\

If $G$ is a CAT(0) group acting geometrically on the CAT(0) space $X$, and there exists a closed convex $H$-invariant subset of $X$ on which $H$ acts cocompactly, $H$ is called \textit{convex}.

With this terminology our theorem becomes: \textit{Any quasiconvex subgroup of $G=F_m\times\mathbb{Z}^n$ is convex with respect to the usual action of $G$ on Tree$\times\mathbb{R}^n$}.

In the course of proving this result, we introduce a technique for analyzing the convex hull in certain special class of CAT(0) spaces and derive the following complete description of the quasiconvex subgroup of $F_m\times\mathbb{Z}^n$:\\\\
\textbf{Theorem}:  If $H$ is a quasiconvex subgroup of $F_m\times\mathbb{Z}^n$, then $H$ is virtually of the form $A\times B$, where $A\leq F_m$ is finitely generated and $B\leq\mathbb{Z}^n$.\\\\

\section{\bf Convex Hulls and Quasiconvex Subgroups}\label{Hullssbgrps}
As we mentioned in the previous section, every convex subspace of a CAT(0) space is obviously itself a CAT(0) space with the induced metric. An idea which dates back to Minkowski and Brunn, and which we independently rediscovered, is to construct $conv(Y)$ by means of a sequential process as follows: For $S\subseteq X$, we define $conv^1(S)$ to be the union of all geodesic segments having both endpoints in $S$, or symbolically $conv^1(S)=\bigcup_{s_1,s_2\in S}\left[s_1,s_2\right]$.  Now, we set $conv^0(Y)=Y$ and define recursively $conv^i(Y)= conv^1(conv^{i-1}(Y))$.  This process of "convexification" results in an ascending sequence of subsets of $X$: $Y=conv^0(Y)\subseteq conv^1(Y)\subseteq\ldots\subseteq conv^i(Y)\subseteq\ldots\subseteq conv(Y)\subseteq X$, each of which gets closer to the convex hull of $Y$ in the following sense:
\begin{lem}\label{SeqConv}Let X be a uniquely geodesic space and Y a subspace, then $conv(Y)=\bigcup_{i=0}^{i=\infty}conv^i(Y)$.
\end{lem}
\begin{proof}Obviously, $conv^i(Y) \subseteq conv(Y)$ for all $i$. It is also clear that \\ $\bigcup_{i=0}^{i=\infty}conv^i(Y)$ is convex, hence it equals $conv(Y)$.
\end{proof}
Lemma \ref{SeqConv} was the starting point for our investigation of convex hulls.  Its proof is not difficult and we realized that the result was already attributed to Hermann Brunn in the case when $X$ is a vector space only after we had proved the results on convexity in polygonal complexes below.\\\\
\textbf{Remark:} A set $Y$ is $\nu$-quasiconvex if and only if $conv^1(Y)\subseteq\mathcal{N}_{\nu}(Y)$.\\\\
We now relate the foregoing discussion on convexity with non-positive curvature.  In CAT(0) spaces, as a consequence of the convexity of the metric, one has control over the growth of the sizes of the sets $conv^i(Y)$ as the following result shows:
\begin{lem}\label{GrowthCont}Let $X$ be a CAT(0) space and let $Y\subseteq X$ be a $\nu$-quasiconvex subset of X.  Then, $conv^i(Y)\subseteq N_{i\nu}(Y)$.
\begin{proof}
In this proof we assume that all geodesics are parametrized proportional to arc length and we proceed by induction on $i$.  The starting step $i=1$ is handled by the remark above.  For the inductive step, suppose that $conv^{i-1}(Y)\subseteq N_{(i-1)\nu}(Y)$ and let $x\in conv^i(Y)$.  Then, $x\in\left[x_1,x_2\right]$, where $x_1,x_2\in conv^{i-1}(Y)\subseteq N_{(i-1)\nu}(Y)$.  Let $x_1',x_2'\in Y$ be such that $d(x_j,x_j')<(i-1)\nu$, for $j=1,2$.  By convexity of the CAT(0) metric, $d(\left[x_1,x_2\right](t),\left[x_1',x_2'\right](t))\leq (1-t)d(x_1,x_1')+td(x_2,x_2')<(i-1)\nu$.  This shows that $d(x,conv^1(Y))<(i-1)\nu$ and thus we conclude that $x\in N_{i\nu}(Y)$, as desired.
\end{proof}
\end{lem}

The number $k(Y)=inf\left\{i:conv^i(Y)=conv(Y)\right\}$ is called the Brunn number of $Y$.  Brunn gave a lower and an upper bound for $k(Y)$ in finite-dimensional vector spaces, see \cite{Brn}.  In view of Lemma \ref{GrowthCont}, a uniform bound for the Brunn number over all subsets of a CAT(0) space has important implications from the point of view of cocompactness of group actions:

\begin{lem}\label{QuasCCoco}Suppose that $G$ is a CAT(0) group, i.e. $G$ acts geometrically on the CAT(0) space $X$, and suppose that $H\leq G$ is a $\nu$-quasiconvex subgroup which has the property that $conv(Hx_0)=conv^k(Hx_0)$ for some $x_0\in X$, and $k\in\mathbb{N}$.  Then, $H$ acts cocompactly on $conv(Hx_0)$, and $H$ is therefore a CAT(0) group.
\begin{proof}The proof is a straightforward application of Lemma \ref{GrowthCont}.\\
If $conv(Hx_0)=conv^k(Hx_0)$, then by Lemma \ref{GrowthCont}, $conv(Hx_0)\subseteq\mathcal{N}_{k\nu}(Hx_0)$, and therefore $conv(Hx_0)/H\subseteq\mathcal{N}_{k\nu}(Hx_0)/H$, i.e. $conv(Hx_0)/H\subseteq\mathcal{N}_{k\nu}(x_0)$.  This shows that the quotient $conv(Hx_0)$ is compact, as desired.
\end{proof}
\end{lem}

  Before we proceed to give a bound for the Brunn number for Euclidean spaces, we make the following important observation:
In analyzing the convex hull of an arbitrary set using Lemma \ref{SeqConv}, it suffices to only consider finite sets, which are substantially easier to work with.
\begin{lem}\label{FinSet}Let $X$ be a geodesic space, let $Y\subseteq X$, and let $i\in\mathbb{N}^+$.  If $conv^i(S)=conv(S)$ for every finite subset $S\subseteq Y$, then $conv^i(Y)$ is convex.
\begin{proof}
Let $x,y\in conv^i(Y)$. Then, we can find $a_1,a_2,b_1,b_2\in conv^{i-1}(Y)$ such that $x\in [a_1,a_2]$ and $y\in [b_1,b_2]$.  Similarly, we can find $c_1,c_2,d_1,d_2\in conv^{i-2}(Y)$ such that $a_1\in [c_1,c_2]$ and $a_2\in [d_1,d_2]$, etc. Proceeding recursively, we see that we can find points $x_1,\ldots , x_m\in Y$ such that $x,y\in conv^i(\left\{x_1,\ldots , x_m\right\})=conv(\left\{x_1,\ldots , x_m\right\})$. Hence, $[x,y]\subseteq conv^i(\left\{x_1,\ldots , x_m\right\})\subseteq conv^i(Y)$.
\end{proof}
\end{lem}

As we were unable to procure Brunn's original paper, we present our own proof of the intuitively obvious fact that the Brunn number of any subset of $\mathbb{R}^n$ is less or equal to $n$.  Our proof uses Caratheodory's theorem which we recall below.
\subsection{\bf Convex Hulls in Euclidean Spaces}
\begin{thm}(Caratheodory) If $E$ is a vector space of dimension $d$, then, for every subset $X$ of $E$, every element in the convex hull $conv(X)$ is an affine convex combination of $d+1$ elements of $X$.
\begin{proof}See Proposition 5.2.3 in \cite{ConvNonp}.
\end{proof}
\end{thm}
\begin{lem}\label{FinRn}For any finite set $S\subseteq\mathbb{R}^n$, $conv^n(S)=conv(S)$.
\end{lem}
\begin{proof}By Caratheodory's theorem, $conv(S)=\bigcup conv(\left\{s_1,...,s_{n+1}\right\})$, where the union is taken over all $s_1,...,s_{n+1}\in S$.  Therefore, it suffices to show that for any set $\left\{s_1,...,s_{n+1}\right\}\subseteq\mathbb{R}^n$, $conv^n(\left\{s_1,...,s_{n+1}\right\})=conv(\left\{s_1,...,s_{n+1}\right\})$.  Consider the points $e_1,...,e_{n+1}\in\mathbb{R}^{n+1}$.  Their convex hull is the standard $n$-simplex $\Delta_n$ in $\mathbb{R}^{n+1}$.  Suppose $conv^{i-1}(\left\{e_1,...,e_{n+1}\right\})$ contains all the $(i-1)$-faces of $\Delta_n$, then $conv^i(\left\{e_1,...,e_{n+1}\right\})$ contains all joins of the form $join\left\{F,e_j\right\}$, $1\leq j\leq n+1$, where $F$ is an $(i-1)$-face.  But all of the $i$-faces are joins of this form.  By induction, $conv^i\left\{e_1,...,e_{n+1}\right\}$ contains all the $i$-faces.  Hence, $conv^n(\left\{e_1,...,e_{n+1}\right\})$ contains and therefore equals $\Delta_n$.  Now, let $\phi$ be the affine map which sends $e_i$ to $s_i$.  This map sends lines to lines, therefore\\ $\phi(conv^i(\left\{e_1,...,e_{n+1}\right\}))\subseteq conv^i(\left\{s_1,...,s_{n+1}\right\})$.  Now, $\phi(conv^n(\left\{e_1,...,e_{n+1}\right\}))$ is convex, contains $\left\{s_1,...,s_{n+1}\right\}$, and is contained in $conv^n(\left\{s_1,...,s_{n+1}\right\})$.  Therefore,\\$conv(\left\{s_1,...,s_{n+1}\right\})=conv^n(\left\{s_1,...,s_{n+1}\right\})$, as desired.
\end{proof}
Combining Lemma \ref{FinSet} and Lemma \ref{FinRn}, we obtain the desired bound on the Brunn number:
\begin{cor}\label{BrunnRn}For any subset $Y\subseteq\mathbb{R}^n$, $conv^n(Y)=conv(Y)$.  Therefore, $k\leq n$ for any subset of $\mathbb{R}^n$.
\end{cor}
We give a straightforward application of Corollary \ref{BrunnRn} to Hilbert geometries.  First, let us recall some basic definitions.  Let $\Omega$ be a bounded open convex subset of a Euclidean space.  Given any two distinct points $x,y\in\Omega$, we define the \textit{Hilbert distance} between $x$ and $y$ to be $H_{\Omega}(x,y)=\log{\left(\dfrac{|x-x^{'}||y-y^{'}|}{|y-x^{'}||x-y^{'}|} \right)}$, where $x^{'}$ denotes the point of intersection of the ray from $x$ through $y$ and the boundary of $\Omega$, and $y^{'}$ denotes the point of intersection of the ray from $y$ to $x$ and the boundary of $\Omega$.  Naturally, we define $H_{\Omega}(x,x)=0$ for any $x\in\Omega$.  It is a well-known fact that $H_{\Omega}$ is a metric on $\Omega$, and that each affine segment is a geodesic segment for the metric $H_{\Omega}$, see Theorem 5.6.7 in \cite{ConvNonp}, also Chapter 5.6 in \cite{ConvNonp}.
\begin{cor}In any uniquely geodesic Hilbert geometry $(\Omega,H_{\Omega})$, the Brunn number of any subset is bounded above by the dimension of the underlying Euclidean space.
\begin{proof}Since each affine segment in the ambient Euclidean space is a geodesic for $H_{\Omega}$ and vice versa, and since $(\Omega,H_{\Omega})$ is assumed to be uniquely geodesic, the convex hull of a finite set of points in $\Omega$ coincides with its convex hull in the ambiant Euclidean space.  Therefore, Corollary \ref{BrunnRn} immediately applies.
\end{proof}
\end{cor}
Unfortunately, obtaining a bound on the Brunn number in an arbitrary CAT(0) space in the absence of any restrictions on the subsets in consideration is impossible since the class of CAT(0) spaces is too large, as the following example shows:\\\\
\textbf{Example:}  Let $\mathcal{H}$ be the infinite dimensional real Hilbert space $l^2(\mathbb{R})$.  Recall that the elements of $\mathcal{H}$ are sequences of real numbers $\textbf{x}=(x_1,x_2,...)$ such that $\sum x_i^2<\infty$, and that $(\textbf{x},\textbf{y})=\sum x_iy_i$ defines an inner product on $\mathcal{H}$, where $\textbf{x}=(x_1,x_2,...),\textbf{y}=(y_1,y_2,...)\in\mathcal{H}$.  This inner product induces a norm on $\mathcal{H}$ called the $l^2$-norm, under which $\mathcal{H}$ becomes a complete uniquely geodesic metric space.  Further, it is obvious that $\mathcal{H}$ is a CAT(0) space as well, since any three non-colinear points lie in a 2-dimensional affine subspace of $\mathcal{H}$ isometric to the Euclidean plane, and thus the CAT(0) comparison inequality holds trivially.  Now, let $S=\left\{\textbf{e}_1, \textbf{e}_2,...,\textbf{e}_l,...\right\}$, where $\textbf{e}_j=(0,0,...,0,1,0,...0,...)$, with a $1$ in the $j$-th position only. 
An easy proof by induction shows that for each $i\in\mathbb{N}$, $conv^i(S)$ is a countable union of finite dimensional simplices in $\mathcal{H}$, and therefore each $conv^i(S)$ has finite Hausdorff dimension.  On the other hand, it is obvious that $conv(S)$ is infinite dimensional, hence we conclude that $conv^i(S)\neq conv(S)$ for any $i\in\mathbb{N}$.\\\\
We are, however, able to bound the Brunn number in certain "planar" piecewise Euclidean CAT(0) polygonal complexes.  Obtaining this bound is the subject of the discussion below.  Since the remainder of the paper does not rely on the results in Section \ref{CHCP}, the reader may safely omit it. 
\subsection{\bf Convex Hulls in CAT(0) Planes}\label{CHCP}
\begin{defn}A \textit{CAT(0) plane} is a simply connected piecewise Euclidean polygonal complex $X$ with $Shapes(X)$ finite, such that $Lk(v)$ is isometric to a circle of length $\geq 2\pi$ for every vertex $v\in X$.
\end{defn}
Recall that if $X$ is a polyhedral complex, $Shapes(X)$ denotes the set of isometry classes of polyhedra, and for any $v\in X^{(0)}$, $Lk(v)$ denotes the equivalence classes of geodesics issuing at $v$ where two geodesics are considered equivalent if they make Alexandrov angle zero.  A reformulation of non-positive curvature, in dimension 2, in more combinatorial terms due to Gromov is the following result:
\begin{thm}\label{GrLkCond}(Gromov) A polygonal complex $X$ with $Shapes(X)$ finite is CAT(0) if and only if $Lk(v)$ contains no embedded circle of length less than $2\pi$.
\end{thm}
For an in-depth discussion on polyhedral complexes, as well as a proof of Theorem \ref{GrLkCond}, which is a special case of \textit{Gromov's Link Condition}, we refer the reader to Chapter II.5 in \cite{BH}.
We immediately note that in view of Theorem \ref{GrLkCond}, a CAT(0) plane is a CAT(0) metric space.  In this section, we prove that the Brunn number of any subset of a CAT(0) plane is $\leq2$.  The proof of this intuitively "obvious" result turned out to be surprisingly technical.  Our proof employs a local-to-global technique which makes use of the Cartan-Hadamard theorem.  The idea is that under mild hypotheses, convexity on the small scale implies global convexity.  Before we are able to state and prove the promised results, we need to transpose some familiar definitions to the "small scale":

\begin{defn}Let $(X,d)$ be a metric space.

\begin{enumerate}
	\item The metric $d$ is said to be \textit{locally convex} if every point in $X$ has a neighborhood in which the induced metric is convex.
	\item The metric space is said to be \textit{locally CAT(0)} if every point in $X$ has a convex neighborhood $\mathcal{U}$ with the property that 											$(\mathcal{U},d)$ is a CAT(0) metric space.
	\item Let $f:X\rightarrow Y$ be a map between two metric spaces.  We shall say that $f$ is a \textit{local isometry} if every point $x\in X$ has a neighborhood 					$\mathcal{U}$, such that $f$ restricted to $\mathcal{U}$ is an isometric embedding.
\end{enumerate}

\end{defn}

\begin{thm}\label{CartanHadamard}(Cartan-Hadamard) Let $(X,d)$ be a complete connected metric space.  If the metric on $X$ is locally convex, then the induced length metric on the universal covering $\widetilde{X}$ is convex.  In particular, there is a unique geodesic segment joining each pair of points in $\widetilde{X}$.  Further, if $X$ is a locally CAT(0) space, then $\widetilde{X}$ with the induced length metric is a CAT(0) space and the covering map $p:\widetilde{X}\rightarrow X$ is a local isometry.
\begin{proof}See Theorem 4.1 in \cite{BH}.
\end{proof}
\end{thm}
In view of the Cartan-Hadamard theorem above, we shall say that a subset $Y$ of the CAT(0) space $X$ is \textit{locally convex} if every point of $Y$ has a convex neighborhood $\mathcal{U}\subseteq X$ such that $\mathcal{U}\cap Y$ is convex.  First, we show that $conv^i(S)$ is compact for every compact set $S$:
\begin{lem}\label{Cpctconv}Let $X$ be a CAT(0) space, and let $S\subseteq X$ be a compact set.  Then, $conv^i(S)\subseteq X$ is compact for every $i\in\mathbb{N}$.
\begin{proof}The proof is by induction on $i$, the case $i=0$ being trivial.  Note that we have an obvious map $\varphi :conv^{i-1}(S)\times conv^{i-1}(S)\times I\rightarrow conv^i(S)$ given by $\varphi(x,y,t)=[x,y](t)$. Since in any CAT(0) space geodesics vary continuously with endpoints, see Proposition 2.2 in Chapter II.2 of \cite{BH}, the map $\varphi$ is a continuous surjection which maps the compact set $conv^{i-1}(S)\times conv^{i-1}(S)\times I$ to $conv^i(S)$ thus proving our claim.
\end{proof} 
\end{lem}
The following proposition opens the door to local-to-global convexity arguments.  It has recently been brought to the attention of the author that a recent paper of Bux and Witzel also proves a similar result in the slightly more general context of CAT($\kappa$) spaces, see Theorem 1.10 in \cite{KU}.
\begin{pro}\label{FinSubsets}Let $X$ be a complete and connected CAT(0) space.  Then, $conv^n(Y)=conv(Y)$ for all $Y\subseteq X$ if and only if $conv^n(S)$ is locally convex for every finite subset $S\subseteq X$.
\begin{proof}The forward direction is trivial, and so we only prove the converse.  Suppose that $conv^n(S)$ is locally convex for every finite subset $S\subseteq X$. As Lemma \ref{Cpctconv} shows that $conv^n(S)$ is a compact and therefore complete, connected, and locally convex subset of a CAT(0) space, Theorem \ref{CartanHadamard} tells us that the universal cover $\widetilde{conv^n(S)}$ endowed with the length metric is a CAT(0) space, and that the covering map $p:\widetilde{conv^n(S)} \rightarrow conv^n(S)$ is a local isometry.  Let $x,y\in conv^n(S)$ and choose any $\tilde{x} \in p^{-1}(x)$, and $\tilde{y} \in p^{-1}(y)$. Let $\alpha(t)$ be the unique geodesic in $\widetilde{conv^n(S)}$ joining $\tilde{x}$ to $\tilde{y}$.  Then, because $conv^n(S)$ is compact, a simple argument using the Lebesgue covering lemma shows that $p\circ \alpha(t)$ is a local geodesic in $X$ joining $x$ to $y$. Since in a CAT(0) space any local geodesic is a geodesic, see Proposition 1.4 in Chapter II.1 of \cite{BH}, we see that $p\circ \alpha$ is the geodesic in $X$ joining $x$ to $y$. But the image of $p\circ \alpha$ is contained in $conv^n(S)$. This shows that $conv^n(S)$ is convex, and Lemma \ref{FinSet} yields the desired conclusion.
\end{proof} 
\end{pro}
In the course of proving Proposition \ref{CAT0Plane} we will encounter the phenomenon of \textit{bifurcating} geodesics.  We shall call geodesics which coincide up to a point of \textit{bifurcation} or divergence, bifurcating geodesics.  The key fact we shall need is that in a CAT(0) plane, at the point of bifurcation, the geodesic which splits can be extended in an infinite number of ways.  To prove this result, we need the following reformulation of the CAT(0) condition:
\begin{lem}\label{Alexleqcomp}Let $X$ be a geodesic metric space.  Then, $X$ is a CAT(0) space if and only if the Alexandrov angle between the sides of any geodesic triangle in $X$ with distinct vertices is no greater than the angle between the corresponding sides of its comparison triangle in $\mathbb{E}^2$.
\begin{proof}See Proposition 1.7(4) in Chapter II.1 of \cite{BH}.
\end{proof}
\end{lem}
\begin{lem}\label{GeodBroom}Let $X$ be a CAT(0) plane and let $a_1,a_2,b,q\in X$ be such that $\left[a_1,b\right]\cap\left[a_2,b\right]=\left[q,b\right]$.  Then, for any $c\in\left[a_1,a_2\right]$, the concatenation of the geodesic segment $\left[c,q\right]$ and $\left[q,b\right]$ is a geodesic segment.
\begin{proof}Since $X$ is a CAT(0) plane, $Lk(q)$ is a circle of length at least $2\pi$.  Now, for any $c\in\left[a_1,a_2\right]$, the distance in $Lk(q)$ between the directions determined by $\left[c,q\right]$ and $\left[q,b\right]$ is at least $\pi$, hence the Alexandrov angle $\angle_{q}\left(c,b\right)=\pi$.  By Lemma \ref{Alexleqcomp}, $\angle_{q}\left(c,b\right)\leq\overline{\angle}_{\overline{q}}\left(\overline{c},\overline{b}\right)$, hence $\overline{\angle}_{\overline{q}}\left(\overline{c},\overline{b}\right)=\pi$.  Now, we conclude that $d(\overline{c},\overline{b})=d(\overline{c},\overline{q})+d(\overline{q},\overline{b})$.  Since the distances between the vertices of $\Delta(a,b,c)$ are the same as the distances between the vertices of the comparison triangle $\overline{\Delta}(\overline{a},\overline{b},\overline{c})$, we have shown that $d(c,b)=d(c,q)+d(q,b)$ thus proving that the concatenation of $\left[c,q\right]$ and $\left[q,b\right]$ is indeed a geodesic segment.
\end{proof}
\end{lem}
We are now ready to show that the Brunn number for any subset of a CAT(0) plane is at most equal to 2.  The seemingly technical proof of Proposition \ref{CAT0Plane} below has a simple intuitive idea: in a CAT(0) plane, for any set $S$ the geodesics of $conv^2(S)$, having their endpoints in $conv^1(S)$, can be "swung" by sliding their endpoints thus sweeping out small convex neighborhoods around any point.  In view of our local-to-global technique in Proposition \ref{FinSubsets}, this establishes convexity of $conv^2(S)$.
\begin{pro}\label{CAT0Plane}("The Swinging Geodesics") Let $X$ be a CAT(0) plane.  Then the Brunn number of any subset $Y\subseteq X$ is at most 2.
\begin{proof}In view of Proposition \ref{FinSubsets}, we only need to show that for every finite subset $S\subseteq X$, $conv^2(S)$ is locally convex at every point $p\in conv^2(S)$.  Suppose to the contrary that $conv^2(S)$ is not locally convex at some $p\in conv^2(S)$.  Then, given any convex neighborhood $\mathcal{U}$ of $p$ in $X$, since $conv^2(S)$ is compact by Lemma \ref{Cpctconv}, we can find a geodesic $\gamma:\left[0,1\right]\rightarrow\mathcal{U}$ such that $\gamma(0),\gamma(1)\in conv^2(S)$ and $\gamma(t)\notin conv^2(S)$ for all $0<t<1$.  Since the image of $\gamma$ is not contained in $conv^2(S)$, at least one of $\gamma(0)$ or $\gamma(1)$ is not in $conv^1(S)$.  Without loss of generality, we may assume that $\gamma(0)=x_0\notin conv^1(S)$; note that we may have $x_0=p$.  Then, there exist points $x_1,x_2,y_1,y_2\in S$ such that $x_0\in\left[\left[x_1,y_1\right](t_1),\left[x_2,y_2\right](t_2)\right]$ for some $t_1,t_2\geq 0$.  Consider the 1-parameter family of geodesics $t'\mapsto\left[\left[x_1,y_1\right](t'),\left[x_2,y_2\right](t_2)\right]$.  If $x_0\in\left[\left[x_1,y_1\right](t'),\left[x_2,y_2\right](t_2)\right]$ for all $t'\geq t_1$ (or all $t'\leq t_1$), then $x_0$ lies on a geodesic having one endpoint, say $s$, in $S$.  In this case we consider the family $t'\mapsto\left[s,\left[x_2,y_2\right](t')\right]$.  If $x_0\in\left[s,\left[x_2,y_2\right](t')\right]$ for all $t'\geq t_2$ (or all $t'\leq t_2$), then $x_0\in conv^1(S)$, which is a contradiction.  Therefore, without loss of generality, we may assume that $x_0\notin\left[\left[x_1,y_1\right](t'),\left[x_2,y_2\right](t_2)\right]$ for some values of $t'$ greater than $t_1$ and also for some values of $t'$ less than $t_1$.  As $X$ has no free edges (i.e. edges which belong to only one polygon), $\gamma$ may be extended to geodesic, which by abuse of notation we shall also denote by $\gamma$, $\gamma:\left[-\epsilon,1\right]\rightarrow X$ such that this extension of $\gamma$ intersects $\left[\left[x_1,y_1\right](t_1),\left[x_2,y_2\right](t_2)\right]$ only at $x_0$.  Because $\gamma(t)\notin conv^2(S)$ for $t>0$ and because geodesics vary continuously with endpoints, we can find $t_0'>0$ and $t_0''>0$ such that $t_0'<t_1<t_0''$ and $-\epsilon<-\epsilon'<0$ such that $\left[\left[x_1,y_1\right](t_0'),\left[x_2,y_2\right](t_2)\right]$ and $\left[\left[x_1,y_1\right](t_0''),\left[x_2,y_2\right](t_2)\right]$ intersect at $\gamma(-\epsilon')$, and such that neither of these geodesic segments passes through $x_0$.  Because of uniqueness of geodesics in $X$, the segments $\left[\left[x_1,y_1\right](t_0'),\left[x_2,y_2\right](t_2)\right]$ and $\left[\left[x_1,y_1\right](t_0''),\left[x_2,y_2\right](t_2)\right]$ must coincide up to some point $q$ which is their point of bifurcation.  Then, Lemma \ref{GeodBroom} shows that for every $t$ between $t_0'$ and $t_0''$, the concatenation of the geodesic segments $\left[\left[x_2,y_2\right](t_2),q\right]$ and $\left[q,\left[x_1,y_1\right](t)\right]$ is a geodesic.  In particular, the concatenation $\left[\left[x_2,y_2\right](t_2),q\right]\cup\left[q,\left[x_1,y_1\right](t_1)\right]$ is the geodesic segment $\left[\left[x_1,y_1\right](t_1),\left[x_2,y_2\right](t_2)\right]$.  Since $x_0\in\left[\left[x_1,y_1\right](t_1),\left[x_2,y_2\right](t_2)\right]$, we must have that either $x_0\in\left[\left[x_2,y_2\right](t_2),q\right]$ or $x_0\in\left[q,\left[x_1,y_1\right](t_1)\right]$, since $q$ is their only point of intersection.  But if $x_0\in\left[\left[x_2,y_2\right](t_2),q\right]$, then $x_0\in\left[\left[x_1,y_1\right](t_0'),\left[x_2,y_2\right](t_2)\right]$, because the former is a subsegment of the latter, which contradicts the assumption that $x_0\notin\left[\left[x_1,y_1\right](t_0'),\left[x_2,y_2\right](t_2)\right]$.  Therefore, we must have that $x_0\in\left[q,\left[x_1,y_1\right](t_1)\right]$, and $x_0\neq q$.  Now, any extension of the geodesics segment $\left[\left[x_1,y_1\right](t_1),\left[x_2,y_2\right](t_2)\right]$ to a biinfinite geodesic divides the CAT(0) plane into two convex half-subspaces each homeomorphic to the upper half-plane, and the geodesic segments $\left[x_1,y_1\right](t_0'),\left[x_2,y_2\right](t_2)$ and $\left[x_1,y_1\right](t_0''),\left[x_2,y_2\right](t_2)$ having started to diverge at $q$, each enter a different convex half-space, which would make it impossible for them to intersect at $\gamma(-\epsilon')$.  This is a contradiction, and we have thus proved the claim. 
\end{proof}
\end{pro}

\section{\bf Free$\times$Free-abelian Groups}\label{FreeFreeab}
We shall need some facts about the isometries of CAT(0) spaces which the reader may recognize as having their origin in hyperbolic geometry.  Let $X$ be a CAT(0) space, and let $\gamma\in Isom(X)$ be an isometry of $X$.  We define the translation length of $\gamma$ to be $|\gamma|=\inf\left\{d(\gamma\cdot x,x):x\in X\right\}$ and we define $Min(\gamma)=\left\{x\in X: d(\gamma\cdot x,x)=|\gamma|\right\}$.  According to whether $|\gamma|$ is achieved or not and whether it is non-zero, $\gamma$ falls in one of three classes: \textit{elliptic, hyperbolic}, and \textit{parabolic}:\\
\begin{itemize}
	\item\label{elliptic} $\gamma\in Isom(X)$ is called elliptic if $\gamma$ has a fixed point;
	\item\label{elliptic} $\gamma\in Isom(X)$ is called hyperbolic if $|\gamma|$ is realized for some $x\in X$;
	\item\label{elliptic} $\gamma\in Isom(X)$ is called parabolic is $|\gamma|$ is not realized, i.e. $Min(\gamma)$ is empty.\\
\end{itemize}
An isometry $\gamma\in Isom(X)$ is hyperbolic if and only if there is a biinfinite geodesic $c:\mathbb{R}\rightarrow X$ on which $\gamma$ acts by translation, i.e. $\gamma\cdot c(t)=c(t+a)$ for some $a>0$, see Theorem 6.8 in Chapter II.6 in \cite{BH}.  The biinfinite geodesic $c$ is called an \textit{axis} for $\gamma$.  It is a well-known fact due to Tits that every non-trivial element $f$ of $F_m$ has a unique axis in $T_{2m}$, namely the set of vertices $v\in T_{2m}$ such that $d(v,f\cdot v)=|f|$ together with the edges that join them, see Lemmas 3.1 and 3.2 of \cite{HerHub}, also \cite{Serre} p. 63.  We are now ready to prove the following technical lemma which is an essential ingredient in the proof of Theorem \ref{Cls}:
\begin{lem}\label{ComVert}Suppose $f,g\in F_m$ do not have the same axis of translation in $T_{2m}$ and suppose that $p_0$ lies on the axis for $g$ in $T_{2m}$.  Then, there exists a vertex $v\in T_{2m}$, and a positive integer $k_0$, such that $\left[g^l\cdot p_0,fg^l\cdot p_0\right]$ passes through $v$ for every $l>k_0$.
\begin{proof}Let $a_g$ be the axis for $g$ passing through $p_0$.  Since $a_g$ and its translate $f\cdot a_g$ are convex, so is their intersection $f\cdot a_g\cap a_g$.  Therefore, $f\cdot a_g\cap a_g$ is isometric to either a geodesic line, a geodesic ray, i.e. an isometrically embedded copy of a half-line, or to a geodesic segment.  First, we show that the intersection must be, in fact, a (compact) geodesic segment.

If $f\cdot a_g\cap a_g$ is a geodesic line, then $f\cdot a_g\cap a_g=a_g$, and $f\cdot a_g=a_g$.  Now, since $f$ has no fixed points, $f$ must act by translation on $a_g$, hence $g$ and $f$ have a common axis contrary to assumption.

Now, suppose that $f\cdot a_g\cap a_g$ is isometric to a geodesic ray.  Then, there is $l\in\mathbb{Z}$ such that $fg^l\cdot p_0\in a_g$ and $fg^{l\pm 1}\cdot p_0\in a_g$.  The plus/minus sign refers to whether the powers of $g$ increase or decrease in the direction the geodesic ray which goes to infinity.  Since $d(fg^l\cdot p_0,fg^{l\pm 1}\cdot p_0)=|g|$, $g^{\pm 1}fg^l\cdot p_0=fg^{l\pm 1}\cdot p_0$.  By freeness of the action of $F_m$, we conclude that $g^{\pm 1}fg^l=fg^{l\pm 1}$, hence $g^{\pm 1}f=fg^{\pm 1}$.  It is easy to check that the four cases reduce to one of the following relations: either $gf=fg$, i.e $f$ and $g$ commute, or $f^2g=gf^2$ commute.  In either case, Proposition 6.2(2) in Chapter II.6 of \cite{BH} shows that $g$ leaves the axis of $f$ invariant.  As before, this contradicts the assumption that $f$ and $g$ do not have the same axis in $T_{2m}$.

Thus, we conclude that $f\cdot a_g\cap a_g$ is compact and we can therefore find a positive integer $k_0$ such that the geodesic rays $\vec{r_1}=\bigcup_{i\geq k_0}\left[g^i\cdot p_0,g^{i+1}\cdot p_0\right]\subset a_g$ and $\vec{r_2}=\bigcup_{i\geq k_0}\left[fg^i\cdot p_0,fg^{i+1}\cdot p_0\right]\subset f\cdot a_g$ are disjoint.  But then, it is obvious that $l=\vec{r_1}\cup\vec{r_2}\cup\left[g^{k_0}\cdot p_0,fg^{k_0}\cdot p_0\right]$ is a geodesic line and $g^l\cdot p_0\in\vec{r_1}$ and $fg^l\cdot p_0\in\vec{r_2}$, for all $l\geq k_0$.  It is also obvious that if $x\in\vec{r_1}$ and $y\in\vec{r_2}$, $\left[x,y\right]$ must pass through all the vertices on the segment $\left[g^{k_0}\cdot p_0,fg^{k_0}\cdot p_0\right]$.  Now, we can take $v$ to be any one of these vertices, and we are done.
\end{proof}
\end{lem}

Throughout the remainder of this section, $G$ will be the group $F_m\times\mathbb{Z}^n$ and $X$ will stand for the cartesian product of the regular $2m$-valent tree $T_{2m}$ and $\mathbb{R}^n$.  The action of $G$ on $X$ is the product action where $F_{m}$ acts as the group of deck transformations on the universal cover of the wedge of $m$ circles, $T_{2m}$, and $\mathbb{Z}^n$ acts by translation on $\mathbb{R}^n$.  We shall denote the projection maps to the tree factor and the Euclidean factor by $p:T_{2m}\times\mathbb{R}^n\rightarrow T_{2m}$ and $pr_{\mathbb{R}^n}:T_{2m}\times\mathbb{R}^n\rightarrow\mathbb{R}^n$ respectively.\\\\
\textbf{Remark}:  An important fact about the projection maps $p$ and $pr_{\mathbb{R}^n}$ is that they map geodesic segments to geodesic segments, see Proposition 5.3(3) in Chapter I.5 of \cite{BH}.\\\\
The following technical lemma is crucial for the remainder of this paper: 

\begin{lem}\label{1}Let $H=\left\langle f_1z_1,...,f_sz_s\right\rangle$, $f_i\in F_m$, $z_i\in\mathbb{Z}^n$ be a quasiconvex subgroup of $G$ with respect to the usual action of $G$ on $X$ such that not all of the $f_i$ have the same axis of translation in $T_{2m}$.  Then, there exist positive integers $k_1,...,k_s$ such that $H$ contains the subgroup $A=\left\langle z_1^{k_1},...,z_s^{k_s}\right\rangle$.
\end{lem}
\begin{proof}
Let $1\leq i\leq s$, let $j$ be such that $f_i$ and $f_j$ have different axes of translation, and  let $l$ be a positive integer.  Find an axis of translation for $f_iz_i$ whose projection to the Euclidean factor passes through $0\in\mathbb{R}^n$.  This can always be done by translating the Euclidean component of any given axis for $f_iz_i$.  Let $x_0$ be a point on the chosen axis of translation for $f_iz_i$ such that $pr_{\mathbb{R}^n}(x_0)=0$.  Consider the sequences of points $(f_iz_i)^lx_0$ and $(f_jz_j)(f_iz_i)^lx_0$.  Because $f_i$ and $f_j$ have different axes, an application of Lemma \ref{ComVert} to the projection of the geodesic segment $\left[(f_iz_i)^lx_0,(f_jz_j)(f_iz_i)^lx_0\right]$ to $T_{2m}$, shows that there is a vertex $v$ in $T_{2m}$ such that for all large enough $l$, the geodesic segment $\left[(f_iz_i)^lx_0,(f_jz_j)(f_iz_i)^lx_0\right]$ passes through the flat $\left\{v\right\}\times\mathbb{R}^n$.  Let $y_l$ denote the point of intersection of $\left\{v\right\}\times\mathbb{R}^n$ and $\left[(f_iz_i)^lx_0,(f_jz_j)(f_iz_i)^lx_0\right]$.  The orthogonal projection of this geodesic segment to the flat $\left\{v\right\}\times\mathbb{R}^n\cong\mathbb{R}^n$ is the geodesic segment between $z_i^l$ and $z_i^l+z_j$.  Since the geodesic $\left[(f_iz_i)^lx_0,(f_jz_j)(f_iz_i)^lx_0\right]$ intersects its orthogonal projection in the point $y_l$, we see that $y_l\in\left[z_i^l,z_i^l+z_j\right]\subset \left\{v\right\}\times\mathbb{R}^n$, hence $d(y_l,(v,z_i^l))\leq \left\|z_j\right\|$.  The orbit $Hx_0$ is quasiconvex, hence there is $\nu>0$ and $h_l\in H$ such that $d(h_lx_0,y_l)<\nu$.  Then, $d(h_lx_0,z_i^lx_0)\leq d(h_lx_0,y_l)+d(y_l,z_i^l(v,0))+d(z_i^l(v,0),z_i^lx_0)<\nu+\left\|z_j\right\|+d(x_0,(v,0))$.  If $\tau=\nu+\left\|z_j\right\|+d(x_0,(v,0))$, then $B_{\tau}(h_lx_0)\cap B_{\tau}(z_i^lx_0)\neq\emptyset$, or $B_{\tau}(h_l^{-1}z_i^lx_0)\cap B_{\tau}(x_0)\neq\emptyset$, for all $l$.  Because the action of $G$ is proper, $h_l^{-1}z_i^l=g\in G$ for infinitely many values of $l$.  Then, for some $k$,$l$ we have $z_i^{l-k}=h_lh_k^{-1}\in H$.  Setting $k_i=l-k$, we obtain $z_i^{k_i}\in H$.
\end{proof}
Let $V$ denote the real span of the vectors $z_1^{k_1},...,z_k^{k_s}$.  We then, have the following: 
\begin{lem}\label{2}
With the same notation as in Lemma \ref{1}, the convex hull of $Hx_0$ equals $conv(p(Hx_0))\times V$, for any $x_0\in\mathbb{R}^n\times T_{2m}$. 
\end{lem}
\begin{proof}First, we note that the projection maps $p,pr_{\mathbb{R}^n}$ commute with the operation of forming the convex hull.  That is, $p(conv(Hx_0))=conv(p(Hx_0))$, and similarly for $pr_{\mathbb{R}^n}$.  Let us show this for the projection map $p$.  We begin by making the observation that $p(conv^1(S))=conv^1(p(S))$ for any set $S$, since $p$ maps the geodesic segment connecting two points to the geodesic segment connecting their images.  Therefore, we have  $p(conv(Hx_0))=p\left(\bigcup_i conv^i(Hx_0)\right)=\bigcup_i p(conv^i(Hx_0))=\bigcup_i conv^i(p(Hx_0))=conv(p(Hx_0))$.

Now, we proceed with the proof of the lemma.

'$\subseteq$':  Without loss of generality, we may assume that $pr_{\mathbb{R}^n}(x_0)=0$.  Clearly, $conv(Hx_0)\subseteq p(conv(Hx_0))\times pr_{\mathbb{R}^n}(conv(Hx_0))$, which after commuting the projection maps past $conv$ gives us the desired inclusion.

'$\supseteq$':  Let $x\in conv(p(Hx_0))\times V$.  Let $y\in conv(Hx_0)$ be such that $p(y)=p(x)$.  Note that because $H$ contains powers of the Euclidean translations $z_1,...,z_k$, the projection of the convex hull of the orbit $Hx_0$ to the Euclidean factor will equal $V$.  Also, $conv(Hx_0)\supseteq V\cdot y$, as $conv(Hx_0)$ is stable under the action of $V$ by translations on the second factor.  
Hence, we can write $x=w\cdot y$, for some $w\in V$, so that $x\in conv(Hx_0)$.
\end{proof}
\begin{lem}\label{3}Let $H$ be as in Lemma \ref{1}.  Then, the group $H$ acts cocompactly on its convex hull.  In particular, the Brunn number of the orbit $Hx_0$ is bounded above by $1+\dim(V)$.
\end{lem}
\begin{proof}First, we show that we can write $p(conv(Hx_0))=conv(p(Hx_0))$ as a union of biinfinite geodesic rays $\gamma$, such that any point on $\gamma$ lies between two points in $p(Hx_0)$.  Note that $T=conv(p(Hx_0))\subseteq T_{2m}$ is itself a tree since it is a connected subset of $T_{2m}$, and also that $conv(p(Hx_0))=conv^1(p(Hx_0))$.  Further, at least one vertex in $T$ has more than one edge attached to it, otherwise $T=conv(p(Hx_0))$ would not be connected.  Because $H$ acts transitively on the vertices of $T$, every vertex of $T$ has this property, hence $\partial T=\emptyset$.  Now, we show how to construct the biinfinite geodesics $\gamma$: let $x\in T$.  Since $x\in conv^1(p(Hx_0))$, we can find $h_1,h_2\in H$ such that $x\in \left[p(h_1\cdot x_0),p(h_2\cdot x_0)\right]$.  Next, since $\partial T=\emptyset$, we can extend $\left[p(h_1\cdot x_0),p(h_2\cdot x_0)\right]$ on both ends by at least one edge.  This new geodesic segment contains $\left[p(h_1\cdot x_0),p(h_2\cdot x_0)\right]$ in its interior and is still contained in $T$.  Further, its endpoints lie in $conv^1(p(Hx_0))$, and therefore we can repeat the process and find new geodesic segments with endpoints in $p(Hx_0)$ which contain them.  Repeating this process, we find an ascending chain of geodesic segments with endpoints in $p(Hx_0)$ which contain the point $x$.  The union of these geodesic segments is a biinfinite geodesic $\gamma$ which passes through $x$, and which has the property that any point on $\gamma$ lies between two points in $p(Hx_0)$.  Since $x$ was an arbitrary point in $T$, we conclude that $T$ can indeed be written as a union of all biinfinite geodesics $\gamma$ such that any point on $\gamma$ lies between two points in $p(Hx_0)$.  

Now, $conv(Hx_0)=\bigcup_{\gamma} \gamma\times V$, and $\gamma\times V\cong\mathbb{R}^{1+\dim(V)}$, with $\gamma$ as described above.  Note that $\gamma\times V$ contains the lattices\\ $p(hx_0)\times\mathbb{Z}-span\left\langle z_1^{k_1},...,z_s^{k_s}\right\rangle$, where $h\in H$.  Since any point on $\gamma$ lies between two points $p(h_1),p(h_2)\in p(Hx_0)$, the convex hull of these lattices is all of $\gamma\times V$, and by Corollary \ref{BrunnRn},\\ $conv^{1+\dim(V)}\left(\bigcup (p(hx_0)\times\mathbb{Z}-span\left\langle z_1^{k_1},...,z_s^{k_s}\right\rangle)\right)=\gamma\times V$, where the union on the left-hand side is taken over all $h\in H$ such that $p(hx_0)\in\gamma$.  Finally,\\
$conv^{1+\dim(V)}(Hx_0)\supseteq conv^{1+\dim(V)}\left(\bigcup_{h\in H} (p(hx_0)\times\mathbb{Z}-span\left\langle z_1^{k_1},...,z_s^{k_s}\right\rangle)\right)\supseteq\bigcup\gamma\times V=\left(\bigcup\gamma\right)\times V=T\times V$, where the last two unions are taken over the biinfinite geodesics $\gamma$ which have the property that any point on $\gamma$ lies between the projections to $T_{2m}$ of two points of $Hx_0$.  But, by Lemma \ref{2}, this last expression is precisely equal to $conv(Hx_0)$, thus proving the main claim of the lemma.
\end{proof}
Combining Lemmas \ref{1}-\ref{3}, we obtain:
\begin{thm}\label{T1} Any subgroup of $F_m\times\mathbb{Z}^n$ which is quasiconvex with respect to the usual action of $F_m\times\mathbb{Z}^n$ on $T_{2m}\times\mathbb{R}^n$ acts cocompactly on the convex hull of any of its orbits.
\end{thm}
\begin{proof}Lemmas \ref{1}-\ref{3} take care of the case when for each $i$ there is $j$ such that such that $f_i$ and $f_j$ have different axes of translation.  If $H=\left\langle f^{k_1}z_1,...,f^{k_s}z_s\right\rangle$, $f\in F_{m}$, then the orbit $Hx_0$ is contained in a single flat $a_f\times V$ isometric to $\mathbb{R}^{1+\dim(V)}$, where $a_f$ is a common axis for all $f^{k_i}$, and $x_0$ is on a common axis for all the $f_iz_i$.  Hence, $conv(Hx_0)=conv^{1+\dim(V)}(Hx_0)$, which shows cocompactness of the action of $H$.  In either of the cases $H=\left\langle f_1,...,f_s\right\rangle$ or $H=\left\langle z_1,...,z_s\right\rangle$, the conclusion is again trivially true.  In the former case $conv(Hx_0)=conv^1(Hx_0)$, while in the latter $conv(Hx_0)=conv^s(Hx_0)$.  Now, cocompactness immediately follows from Lemma \ref{QuasCCoco}.
\end{proof}
In the course of proving the theorem, we have the essential ingredients for the following interesting result:
\begin{cor}\label{C1}If $H$ is a subgroup of $F_m\times\mathbb{Z}^n$ which is quasiconvex with respect to the usual action of $F_m\times\mathbb{Z}^n$ on $T_{2m}\times\mathbb{R}^n$, and whose image in $F_m$ under the natural projection $F_m\times\mathbb{Z}^n\rightarrow F_m$ has rank greater than $1$, then $H$ is virtually of the form $A\times B$, where $A\leq F_m$ is finitely generated and $B\leq\mathbb{Z}^n$.
\begin{proof}Since the rank of the projection of $H$ is greater than $1$, given any $g=fz\in H$, we can find an element $f'z'\in H$, such that $f$ and $f'$ have different axes.  Therefore, the last line of the proof of Lemma \ref{1} shows that for $g=fz\in H$, there exists $t$ such that $z^t\in H$, and hence $f^t\in H$.  Let $A=F_m\cap H$ and $B=\mathbb{Z}^n\cap H$.  Then, $g^t\in AB$.  On the other hand, $\left[H,H\right]\subseteq A$, and also $AB$ is normal in $H$.  Hence, we see that $H/AB$ is a finitely generated, torsion, abelian group, and is therefore finite, thus proving the claim.
\end{proof}
\end{cor}
On the other hand, it is easy to show that any subgroup which is virtually of the form $A\times B\subseteq F_m\times\mathbb{Z}^n$, where $A\subseteq F_m$ is finitely generated, and $B\subseteq\mathbb{Z}^n$, is quasiconvex:
\begin{pro}\label{Cls}Let $H$ be a subgroup of $G=F_m\times\mathbb{Z}^n$.  If $H$ is virtually of the form $A\times B$, where $A\subseteq F_m$ is finitely generated, and $B\subseteq\mathbb{Z}^n$, then $H$ is quasiconvex with respect to the standard action of $G$ on $T_{2m}\times\mathbb{R}^n$.
\begin{proof}First, we prove that any subgroup of the form $A\times B$, with $A$ and $B$ as above, is quasiconvex.  Let $A=\left\langle f_1,...,f_s\right\rangle$ and $B=\left\langle z_1,...,z_t\right\rangle$, and let $p_1\in T_{2m}$, $p_2\in\mathbb{R}^n$.  We shall show that $conv^1(A\times B\cdot(p_1,p_2))\subseteq\mathcal{N}_{\nu}(A\times B\cdot (p_1,p_2))$.  First, we observe that for any $x\in T_{2m}$ and $y\in\mathbb{R}^n$, we have $\mathcal{N}_{\nu_1}(x)\times\mathcal{N}_{\nu_2}(y)\subseteq\mathcal{N}_{\sqrt{\nu_1^2+\nu_2^2}}((x,y))$: suppose that $x'\in\mathcal{N}_{\nu_1}(x)$, and $y'\in\mathcal{N}_{\nu_2}(y)$.  Then, $d_{T_{2m}}(x,x')<\nu_1$ and $d_{\mathbb{R}^n}(y,y')<\nu_2$, hence $d_{T_{2m}\times\mathbb{R}^n}((x,y),(x',y'))=\sqrt{(d_{T_{2m}}(x,x'))^2+(d_{\mathbb{R}^n}(y,y'))^2}<\sqrt{\nu_1^2+\nu_2^2}$, and $(x',y')\in\mathcal{N}_{\sqrt{\nu_1^2+\nu_2^2}}((x,y))$.

Again, without loss of generality we may assume that $T_{2m}$ has been metrized so that every edge has length equal to 1.  Now, since $A\leq F_m$ is finitely generated, $A$ is $\nu_1$-quasiconvex, where we can take $\nu_1$ to be the largest of the word lengths of the $f_i$'s.  On the other hand, every subgroup of $\mathbb{Z}^n$ is $\nu_2$-quasiconvex for a large enough $\nu_2$.  Now, suppose that $c:\left[0,1\right]\rightarrow T_{2m}\times\mathbb{R}^n$ is a linearly parametrized geodesic connecting the points $(a_1\cdot p_1,b_1\cdot p_2)$ and $(a_2\cdot p_1,b_2\cdot p_2)$ in $A\times B\cdot(p_1,p_2)$.  By Proposition 5.3(3) in Chapter I.5 of \cite{BH}, $p\circ c$ and $pr_{\mathbb{R}^n}\circ c$ are both linearly prametrized geodesics which connect $a_1\cdot p_1$ with $a_2\cdot p_1$, and $b_1\cdot p_2$ with $b_2\cdot p_2$, respectively.  Since $A$ is quasiconvex in $F_m$ and $B$ is quasiconvex in $\mathbb{Z}^n$, the image of $p\circ c$ is contained in $\mathcal{N}_{\nu_1}(A\cdot p_1)$, and the image of $pr_{\mathbb{R}^n}\circ c$ is contained in $\mathcal{N}_{\nu_2}(B\cdot p_2)$.  Hence, the image of $c$ is contained in $\mathcal{N}_{\nu_1}(A\cdot p_1)\times\mathcal{N}_{\nu_2}(B\cdot p_2)\subseteq\mathcal{N}_{\sqrt{\nu_1^2+\nu_2^2}}(A\times B\cdot (p_1,p_2))$, showing that $A\times B$ is quasiconvex in $F_m\times\mathbb{Z}^n$.

Now, suppose that $H$ is a finite index subgroup of $A\times B$.  As $A\times B$ is quasiconvex, Theorem \ref{T1} shows that the quotient of $conv(A\times B\cdot x_0)$ by $A\times B$ is compact, for any $x_0\in T_{2m}\times\mathbb{R}^n$.  However, $H$ acts on $conv(A\times B\cdot x_0)$, and $conv(A\times B\cdot x_0)/H$ is a finite cover of $conv(A\times B\cdot x_0)/A\times B$.  Therefore, $conv(A\times B\cdot x_0)/H$ is also compact, and we conclude that for some $r>0$, $\mathcal{N}_r(H x_0)\supseteq conv(A\times B\cdot x_0)$.  But $conv^1(H x_0)\subseteq conv(A\times B\cdot x_0)$, which shows that $H$ is quasiconvex, as required. 
\end{proof}
\end{pro}
Combining Corollary \ref{C1} and Proposition \ref{Cls}, we obtain:
\begin{thm}\label{T2} If $H$ is a subgroup of $F_m\times\mathbb{Z}^n$ whose image under the natural projection $F_m\times\mathbb{Z}^n\rightarrow F_m$ has rank greater than $1$, then $H$ is quasiconvex with respect to the usual action of $F_m\times\mathbb{Z}^n$ on $T_{2m}\times\mathbb{R}^n$ if and only if $H$ is virtually of the form $A\times B$, where $A\leq F_m$ is finitely generated and $B\leq\mathbb{Z}^n$.
\end{thm}
Before turning to applications of our results, we would like to mention that bounding Brunn numbers of arbitrary subsets of $T_{2m}\times\mathbb{R}^n$, not just those arising as group orbits of quasiconvex subgroups, is not easy.  In fact, the author actually believes that finding such a bound may not even be possible.
\section{\bf Applications}
We conclude this paper with an application of Theorem \ref{T2} to computational group theory.  In a very recent paper \cite{VentDel}, J. Delgado and E. Ventura have shown that the Finite Index Problem for $F_m\times\mathbb{Z}^n$ is solvable.  The Finite Index Problem for a group $G$ is the following: \textit{Given a finite set of elements $\left\{w_1,...,w_s\right\}$ in $G$, decide whether the subgroup $H=\left\langle w_1,...,w_s\right\rangle$ is of finite index in $G$, and if so, compute the index and a system of right (or left) coset representatives for $H$.}  One of the results in \cite{VentDel} is:
\begin{thm}(\textbf{Delgado, Ventura} \cite{VentDel}) The Finite Index Problem for $F_m\times\mathbb{Z}^n$ is solvable.
\begin{proof}See Theorem 3.4 in \cite{VentDel}.
\end{proof}
\end{thm}
Combining this result with the characterization of quasiconvex subgroups of $F_m\times\mathbb{Z}^n$ given by our theorem \ref{Cls}, Delgado and Ventura establish the existence of an algorithm which decides whether a subgroup of $F_m\times\mathbb{Z}^n$ is quasiconvex or not:
\begin{cor}(\textbf{Delgado, Ventura} \cite{VentDel}) There exists an algorithm which, given a finite list $w_1,...,w_s$ of elements in $F_m\times\mathbb{Z}^n$, decides whether the subgroup $H=\left\langle w_1,...,w_s\right\rangle$ is quasiconvex or not.
\begin{proof}See Corollary 3.7 in \cite{VentDel}.
\end{proof}
\end{cor}

\bibliographystyle{plain}

\end{document}